\documentclass[11pt]{article}

\date{}
\usepackage[utf8]{inputenc}
\usepackage[T1]{fontenc}

\usepackage{graphicx}
\usepackage{tabularx}
\usepackage{url}
\usepackage{amsthm}
\usepackage{amsmath}
\usepackage{amsfonts}

\usepackage{amssymb,epsfig}

\newcommand{\below}{\preceq}
\newcommand{\sbelow}{\prec}
\newcommand{\vect}{\textbf}

\newcommand{\orl}{\vee}
\newcommand{\andl}{\wedge}
\newcommand{\notl}{\overline}

\newtheorem{theorem}{Theorem}
\newtheorem{lemma}{Lemma}

\newtheorem{claim}{Claim}

\newtheorem{conjecture}{Conjecture}

{
\theoremstyle{definition}
\newtheorem{definition}{Definition}
}

\usepackage{color}

\begin{document}

\title{Linear read-once and related Boolean functions\thanks{This paper extends and strengthens the results presented at the 28th International Conference on Algorithmic Learning Theory \cite{Lozin}.}}

\author{
Vadim Lozin\thanks{Mathematics Institute, University of Warwick, UK. Email:  V.Lozin@warwick.ac.uk}
\and 
Igor Razgon\thanks{Department of Computer Science and Information Systems, Birkbeck University of London, UK.  Email: Igor@dcs.bbk.ac.uk}
\and
Viktor Zamaraev\thanks{Department of Computer Science, Durham University, South Road, Durham, DH1 3LE,  UK. Email: viktor.zamaraev@gmail.com}
\and 
Elena Zamaraeva\thanks{Mathematics Institute, University of Warwick, UK. Email: E.Zamaraeva@warwick.ac.uk}
\and 
Nikolai Yu. Zolotykh\thanks{Institute of Information Technology, Mathematics and Mechanics, Lobachevsky State University of Nizhni Novgorod, Russia. Email: Nikolai.Zolotykh@itmm.unn.ru}
}

\maketitle

\begin{abstract}
It is known that a positive Boolean function $f$ depending on $n$ variables has at least $n+1$ extremal points, i.e. minimal ones and maximal zeros.
We show that $f$ has exactly $n+1$ extremal points if and only if it is linear read-once.

The class of linear read-once functions is known to be the intersection of the classes of read-once and threshold functions. 
Generalizing this result we show that the class of linear read-once functions is the intersection of read-once and Chow functions. 
We also find  the set of minimal read-once functions which are not linear read-once and the set of minimal threshold functions which are not linear read-once.
In other words, we characterize the class of linear read-once functions by means of minimal forbidden subfunctions within the universe of read-once and  the universe of threshold functions.

Within the universe of threshold functions the importance of linear read-once functions is due to the fact that they attain the minimum value of the specification number, 
which is $n+1$ for functions depending on $n$ variables. 
In 1995 Anthony et al. conjectured that for all other threshold functions the specification number is strictly greater than $n+1$. 
We disprove this conjecture by exhibiting a threshold non-linear read-once function depending on $n$ variables whose specification number is $n+1$.
\end{abstract}


\section{Introduction}
Linear read-once functions constitute a remarkable subclass of  several   important classes of Boolean functions and appear in the literature frequently under various names. 
In theoretical computer science, they are known as \textit{nested} \cite{AnthonyBrightwellShaweTaylor1995} or \textit{linear read-once} \cite{lists}.
The latter term is justified by the fact that a linear read-once function is a read-once function admitting a Boolean formula that can
be constructed inductively in a ``linear'' fashion. In \cite{lists}, it was also proved that linear read-once functions are equivalent to \textit{$1$-decision lists}.
Later and independently this class was introduced in \cite{Kauffman2003} under the name \textit{nested canalyzing functions},
as a subclass of canalyzing functions, which appear to be important in biological applications \cite{Kauffman1993}.
The significance of nested canalyzing functions in the applications motivated their further theoretical study \cite{NCF2007, NCF2013}.
In particular, \cite{NCF2007} establishes the equivalence between nested canalyzing and \textit{unate cascade functions},
which have been studied in the design of logic circuits and binary decision diagrams.
All mentioned terms refer to the same class of Boolean functions which we call linear read-once.

In \cite{lists}, it was shown that the class of linear read-once functions is the intersection of the classes of read-once and threshold functions. 
Generalizing this result we show that the class of linear read-once functions is the intersection of read-once and Chow functions. 
We also find  the set of minimal read-once functions which are not linear read-once and the set of minimal threshold functions which are not linear read-once.
In other words, we characterize the class of linear read-once functions by means of minimal forbidden subfunctions within the universe of read-once and  the universe of threshold functions.

Within the universe of threshold functions the importance of linear read-once functions is due to the fact that they attain the minimum value of the specification number, 
i.e. of the number of Boolean points that uniquely specify a function in this universe. 
To study the range of values of specification number of threshold functions one can be restricted to positive threshold functions
depending on all their variables,
in which case the functions can be completely specified by their sets of extremal points, i.e. maximal zeros and minimal ones. In other words, the specification number of a  
positive threshold function is upper bounded by the number of its extremal points. For a linear read-once function, these numbers coincide and equal $n+1$. 
In 1995 Anthony et al. \cite{AnthonyBrightwellShaweTaylor1995} conjectured that for all other threshold functions the specification number is strictly greater than $n+1$. 

Our result about the minimal threshold non-linear read-once functions seems to support this conjecture, since all these functions have specification number $2n$.
One more result proved in this paper, which can be viewed as a supporting argument for the conjecture, 
states that a positive function $f$ depending on $n$ variables has exactly $n+1$ extremal points if and only if it is linear read-once. 
Nevertheless, rather surprisingly, we show that the conjecture is not true by exhibiting a positive threshold non-linear read-once function depending on $n$ variables whose specification number is $n+1$. 

The organization of the paper is as follows. All preliminary information related to the topic of the paper, including definitions and notation, is presented in Section~\ref{sec:pre}.
Section~\ref{sec:ex} is devoted to the number of extremal points in positive functions. In Section~\ref{sec:Chow} we show that the class of linear read-once functions is 
the intersection of the classes of read-once and Chow functions, and identify the set of minimal read-once functions which are not linear read-once.
In Section~\ref{sec:threshold}, we characterize the class of linear read-once functions in terms of minimal threshold functions which are not linear read-once and give a counterexample to the conjecture of Anthony et al.

\section{Preliminaries}
\label{sec:pre}

Let $B = \{ 0, 1 \}$. For a \textit{Boolean $n$-dimensional hypercube} $B^n$
we define \textit{sub-hypercube} $B^n(x_{i_1} = \alpha_1, \ldots, x_{i_k} = \alpha_k)$
as the set of all points of $B^n$ for which coordinate $i_j$ is equal to $\alpha_j$
for every $j = 1, \ldots, k$.
For a point $\vect{x} \in B^n$ we denote by $\notl{\vect{x}}$ the point in $B^n$
with $(\notl{\vect{x}})_i = 1$ if and only if $(\vect{x})_i = 0$ for every $i \in [n]$.

For a Boolean function $f=f(x_1, \ldots, x_n)$ on $B^n$, $k \in [n]$, and $\alpha_k \in \{ 0, 1\}$ 
we  denote by $f_{|x_k=\alpha_k}$ the Boolean function on $B^{n-1}$ defined as
follows:
$$
	f_{|x_k=\alpha_k} (x_1, \ldots, x_{k-1}, x_{k+1}, \ldots, x_n) =
	f(x_1, \ldots, x_{k-1}, \alpha_k, x_{k+1}, \ldots, x_n).
$$
For $i_1, \ldots, i_k \in [n]$ and $\alpha_1, \ldots, \alpha_k \in \{ 0,1 \}$ we denote by
$f_{|x_1 = \alpha_1, \ldots, x_k = \alpha_k}$ the function 
$(f_{|x_1 = \alpha_1, \ldots, x_{k-1} = \alpha_{k-1}})_{|x_k = \alpha_k}$.
We say that $f_{|x_1 = \alpha_1, \ldots, x_k = \alpha_k}$ is the \textit{restriction} of 
$f$ to $x_1 = \alpha_1, \ldots, x_k = \alpha_k$.
We also say that a Boolean function $g$ is a \textit{restriction} (or \textit{subfunction}) of a Boolean function $f \in B^n$
if there exist $i_1, \ldots, i_k \in [n]$ and $\alpha_1, \ldots, \alpha_k \in \{ 0,1 \}$ such that
$g = f_{|x_1 = \alpha_1, \ldots, x_k = \alpha_k}$.

A variable $x_k$ is called \textit{irrelevant} for $f$ if $f_{|x_k=1} \equiv f_{|x_k=0}$, i.e.,
$f_{|x_k=1}(\vect{x}) = f_{|x_k=0}(\vect{x})$ for every $\vect{x} \in B^{n-1}$. Otherwise,
$x_k$ is called \textit{relevant} for $f$.
If $x_k$ is irrelevant for $f$ we also say that $f$ \textit{does not depend on} $x_k$.

\subsection{Positive functions and extremal points}

By $\below$ we denote a partial order over the set $B^n$,
induced by inclusion in the power set lattice of the $n$-set.
In other words,  
$\vect{x} \below \vect{y}$ if $(\vect{x})_i = 1$ implies $(\vect{y})_i=1$.
In this case we will say that $\vect{x}$ is \textit{below} $\vect{y}$.
When $\vect{x} \below \vect{y}$ and $\vect{x} \neq \vect{y}$ we will sometimes
write $\vect{x} \sbelow \vect{y}$.

\begin{definition}
A Boolean function $f$ is called \textit{positive} (also known as  \textit{positive monotone} or \textit{increasing}) 
if $f(\vect{x}) = 1$ 
and $\vect{x} \below \vect{y}$ imply $f(\vect{y})=1$.
\end{definition}

For a positive Boolean function $f$, the set of its false points forms a down-set and 
the set of its true points forms an up-set of the partially ordered set $(B^n, \below)$.
We denote by 
\begin{itemize}
\item[$Z^f$] the set of maximal false points,
\item[$U^f$] the set of minimal true points. 
\end{itemize}
We will refer to a point in $Z^f$
as a \textit{maximal zero of $f$} and to a point in $U^f$ as a \textit{minimal one of $f$}.
A point will be called an \textit{extremal point of $f$} if it is either a maximal zero or a minimal one of $f$.
We denote by 
\begin{itemize}
\item[$r(f)$] the number of extremal points of $f$.
\end{itemize}

\subsection{Threshold functions}
\begin{definition}
A Boolean function $f$ on $B^n$ is called a \textit{threshold function} if there exist
$n$ \textit{weights} $w_1, \ldots, w_n \in \mathbb{R}$ and a \textit{threshold} $t \in \mathbb{R}$
such that, for all $(x_1, \ldots, x_n) \in B^n$,
$$
	f(x_1, \ldots, x_n) = 0 \iff \sum\limits_{i=1}^n w_i x_i \leq t.
$$
\end{definition}
The inequality $w_1 x_1 + \ldots + w_n x_n \leq t$ is called \textit{threshold inequality} representing
function $f$. 
The hyperplane $w_1 x_1 + \ldots + w_n x_n = t$ is called \textit{separating hyperplane} 
for the function $f$.
It is not hard to see that there are uncountably many different threshold inequalities 
(and separating hyperplanes)
representing a given 
threshold function, and if there exists an inequality with non-negative weights, then $f$ is a positive function.


Let $k \in \mathbb{N}, k \geq 2$. A Boolean function $f$ on $B^n$ is \textit{$k$-summable}
if, for some $r \in \{ 2, \ldots, k \}$, there exist $r$ (not necessarily distinct)
false points $\vect{x}_1, \ldots, \vect{x}_r$ and $r$ (not necessarily distinct) 
true points $\vect{y}_1, \ldots, \vect{y}_r$ such that 
$\sum_{i=1}^r \vect{x}_i = \sum_{i=1}^r \vect{y}_i$ (where the summation is over $\mathbb{R}^n$).
A function is \textit{asummable} if it is not $k$-summable for all $k \geq 2$.

\begin{theorem} \label{th:asum} {\rm {\cite{E61}}}   
	A Boolean function is a threshold function if and only if it is asummable.
\end{theorem}

It is known (see e.g. Theorem 9.3 in \cite{CH11}) that the class of threshold functions is closed under taking restrictions, i.e. any restriction of a threshold function is again a threshold function.
\subsection{Chow functions}

An important class of Boolean functions was introduced in 1961 by Chow \cite{Chow1961} and is known nowadays as {\it Chow functions}.
This notion can be defined as follows. 

\begin{definition}
The Chow parameters of a Boolean function $f(x_1, \ldots, x_n)$ are the $n+1$
integers $(w_1(f),w_2(f), \ldots ,w_n(f),w(f))$, where $w(f)$ is the number of true points of $f$ and $w_i(f)$ is the number of true points of $f$ where $x_i$ is also true.
A Boolean function $f$ is a Chow function if no other function has the same Chow parameters as $f$.
\end{definition}

The importance of the class of Chow functions is due to the fact that it contains all threshold functions, which was also shown by Chow in \cite{Chow1961}.

\subsection{Read-once, linear read-once and canalyzing functions}

\begin{definition}
A Boolean function $f$ is called \textit{read-once} if it can be represented by a Boolean formula using the operations of conjunction, disjunction, and negation in which every variable appears at most once.
We say that such a formula is a \textit{read-once formula} for $f$.
\end{definition}

\begin{definition}
A read-once function $f$ is \textit{linear read-once} (\textit{lro})
if it is either a constant function, or it can be represented by a \textit{nested formula}
defined recursively as follows:
\begin{enumerate}
	\item both literals $x$ and $\notl{x}$ are nested formulas;
	\item $x \orl t$, $x \andl t$, $\notl{x} \orl t$, $\notl{x} \andl t$ are nested formulas,
	where $x$ is a variable and $t$ is a nested formula that contains neither $x$, nor $\notl{x}$.
\end{enumerate}
\end{definition}

It is not difficult to see that an lro function $f$ is positive if and only if 
a nested formula representing $f$ does not contain negations.  

In \cite{lists}, it has been shown that the class of lro functions is precisely the intersection 
of threshold and read-once functions. 

\begin{definition}
A Boolean function $f = f(x_1, \ldots, x_n)$ is called \textit{canalyzing}\footnote{The notion of canalyzing functions was introduced in \cite{Kauffman1993} and 
is widely used in biological applications of Boolean networks. The lro functions form a subclass of canalyzing functions and are known in this context as nested canalyzing.}
if there exists $i \in [n]$ such that $f_{|x_i=0}$ or $f_{|x_i=1}$ is a constant function.
\end{definition}

It is easy to see that if $f$ is a positive canalyzing function then $f_{|x_i=0} \equiv \vect{\textup{0}}$ or $f_{|x_i=1} \equiv \vect{\textup{1}}$, for some $i \in [n]$.

\subsection{Specifying sets and specification number}

Let $\mathcal{F}$ be a class of Boolean functions of $n$ variables, and let $f \in \mathcal{F}$. 

\begin{definition}
A set of points $S \subseteq B^n$ is a \textit{specifying set for} $f$ in $\mathcal{F}$
if the only function in $\mathcal{F}$ consistent with $f$ on $S$ is $f$ itself. 
In this case we also say that $S$ \textit{specifies} $f$ in the class  $\mathcal{F}$.
The minimal cardinality of a specifying set for $f$ in $\mathcal{F}$ is called the 
\textit{specification number of $f$} (in $\mathcal{F}$) and denoted $\sigma_{\mathcal{F}}(f)$.
\end{definition}
 
Let $\mathcal{H}_n$ be the class of threshold Boolean functions of $n$ variables. 
It was shown in \cite{Hu1965} and later in \cite{AnthonyBrightwellShaweTaylor1995} that the specification number of a threshold function of $n$ variables is at least $n+1$.  

\begin{theorem}\label{th:th-spec}{\rm \cite{Hu1965,AnthonyBrightwellShaweTaylor1995}}
	For any threshold Boolean function $f$ of $n$ variables $\sigma_{\mathcal{H}_n}(f) \geq n+1$. 
\end{theorem}

Also,  in \cite{AnthonyBrightwellShaweTaylor1995} it was shown that  the lower bound is attained for lro functions.

\begin{theorem}\label{th:nested-spec}{\rm \cite{AnthonyBrightwellShaweTaylor1995}}
	For any lro function $f$ depending on all its $n$ variables, $\sigma_{\mathcal{H}_n}(f) = n+1$.
\end{theorem}


\subsection{Essential points}

In estimating the specification number 
of a threshold Boolean function $f \in \mathcal{H}_n$ 
it is often useful to consider essential points of $f$ defined as follows.

\begin{definition} 
A point $\vect{x}$ is \textit{essential} for $f$ (with respect to class $\mathcal{H}_n$), if
there exists a function $g \in \mathcal{H}_n$ such that $g(\vect{x}) \neq f(\vect{x})$ and
$g(\vect{y}) = f(\vect{y})$ for every $\vect{y} \in B^n$, $\vect{y} \neq \vect{x}$.
\end{definition}
 
Clearly, any specifying set for $f$ must contain all essential points for $f$. It turns out that 
the essential points alone are sufficient to specify $f$ in $\mathcal{H}_n$ \cite{C65}.
Therefore, we have the following well-known result.
\begin{theorem}\label{cl:ess-sigma}{\rm \cite{C65}}
	The specification number $\sigma_{\mathcal{H}_n}(f)$ of a function $f \in \mathcal{H}_n$
	is equal to the number of essential points of $f$.
\end{theorem}

The following result is a restriction of Theorem~4 in~\cite{ZolotykhShevchenko1999} (proved for threshold functions of many-valued logic)
to the case of Boolean threshold functions.  


\begin{theorem}\label{th:ess-all}{\rm \cite{ZolotykhShevchenko1999}}
A zero of a threshold function $f$ is essential if and only if
there is separating hyperplane containing it.
\end{theorem}

Thus, the set of all essential zeros (\textit{resp.} ones) of $f \in \mathcal{H}_n$ is the union of all points in $B^n$
	belonging to at least one separating hyperplane for the function $f$ (\textit{resp.} $\overline{f}$).

\subsection{The number of essential points vs the number of extremal points}

It was observed in \cite{AnthonyBrightwellShaweTaylor1995} that in the study of specification number of threshold functions, one can be restricted to positive functions. 
To prove Theorem~\ref{th:nested-spec}, the authors of \cite{AnthonyBrightwellShaweTaylor1995} first showed that for a positive
threshold function $f$ depending on all its variables the set of extremal points
specifies $f$. Then they proved that for any positive lro function $f$ of $n$ relevant variables
the number of extremal points is $n+1$.

In addition to proving Theorem~\ref{th:nested-spec}, the authors of \cite{AnthonyBrightwellShaweTaylor1995} also conjectured
that lro functions are the only  functions with the specification number $n+1$ in the class 
$\mathcal{H}_n$.

\begin{conjecture}\label{con:conjecture}{\rm\cite{AnthonyBrightwellShaweTaylor1995}}
	If $f \in \mathcal{H}_n$ 
	has the specification number $n+1$, then $f$ is linear read-once.
\end{conjecture}

In the next section, we show that this conjecture becomes a true statement if
we replace `specification number' by `number of extremal points'. 
Nonetheless, in spite of this result supporting the conjecture, we conclude the paper with a counterexample 
disproving it.


\section{Positive functions and the number of extremal points}
\label{sec:ex}



The main goal of this section is to prove the following theorem.
\begin{theorem}\label{th:extremal_main}
	Let $f = f(x_1, \ldots, x_n)$ be a positive function with $k \geq 0$ relevant variables. 
	Then the number of extremal points of $f$ is at least $k+1$. Moreover $f$
	has exactly $k+1$ extremal points if and only if $f$ is lro.
\end{theorem}


We will prove Theorem~\ref{th:extremal_main} by induction on $n$. The statement
is easily verifiable for $n=1$. Let $n > 1$ and assume that the theorem is true for
functions of at most $n-1$ variables. In the rest of the section we prove
the statement for $n$-variable functions. Our strategy consists of three major steps.
First, we prove the statement for canalyzing functions in Section~\ref{sec:split}. 
This case includes lro functions.
Then, in Section~\ref{sec:ns1}, we prove the result 
for non-canalyzing functions $f$ such that for each variable   $x_i$  both restrictions 
$f_{|x_i=0}$ and $f_{|x_i=1}$ are canalyzing. Finally, in Section~\ref{sec:ns2}, we consider the case of non-canalyzing
functions $f$ depending on a variable $x_i$ such that at least one of the restrictions 
$f_{|x_i=0}$ and $f_{|x_i=1}$ is non-canalyzing. 
In Section~\ref{sec:str}, we introduce some terminology and prove a preliminary result.


\subsection{A property of extremal points}
\label{sec:str}


We say that a maximal zero (\textit{resp.} minimal one) $\vect{y}$ of $f(x_1, \ldots, x_n)$
\textit{corresponds to a variable} $x_i$ if $(\vect{y})_i = 0$ (\textit{resp.} $(\vect{y})_i = 1$).
It is not difficult to see that for any relevant variable $x_i$, there exists at least one minimal one and 
at least one maximal zero corresponding to $x_i$. 
We say that an extremal point of $f$ {\it corresponds to a set $S$ of variables} if 
it corresponds to at least one variable in $S$. 

\begin{lemma}\label{lem:acyclic}
	For every set $S$ of $k$ relevant variables of a positive function $f$, there exist at least $k+1$ extremal points corresponding to this set.
\end{lemma}
\begin{proof}
Let $S$ be a minimal counterexample and let $P$ be the set of extremal points corresponding to the variables in $S$. 
Without loss of generality we assume that $S$ consists of the first $k$ variables of the function, i.e. 
$S=\{x_1,\ldots,x_k\}$. Due to the minimality of $S$ we may also assume that  $|P|=k$ and for every proper subset $S'$ of $S$ there exist at least $|S'|+1$ extremal points corresponding to $S'$.
This implies, by  Hall's Theorem of distinct representatives \cite{Hall}, that there exists a bijection between $S$ and $P$ 
mapping variable $x_i$ to a point $\vect{a}^i\in P$ {\it corresponding} to $x_i$.

Let $\vect{a}$ be any maximal zero in $P$. We denote by $\vect{b}$ the point which coincides with $\vect{a}$ in all coordinates beyond the first $k$,
and for each $i\in \{1,2,\ldots,k\}$ we define the $i$-th coordinate of $\vect{b}$ to be $1$ if $\vect{a}^i$ is a maximal zero, and to be $0$ if $\vect{a}^i$ is a minimal one.

Assume first that $f(\vect{b})=0$ and let $\vect{c}$ be any maximal zero above  $\vect{b}$ (possibly $\vect{b}=\vect{c}$).
If $(\vect{c})_1=\ldots=(\vect{c})_k=1$, then $\vect{a} \sbelow \vect{c}$, contradicting that $\vect{a}$ is a maximal zero. Therefore, 
$(\vect{c})_i=0$ for some $1\le i\le k$ and hence $\vect{c}$ is a maximal zero corresponding to $x_i\in S$. Moreover, $\vect{c}$ is different from any maximal zero $\vect{a}^j\in P$
because the $j$-th coordinate of $\vect{a}^j\in P$ is 0, while the $j$-th coordinate of $\vect{c}$ is 1.

Suppose now that $f(\vect{b})=1$ and let $\vect{c}$ be any minimal one below  $\vect{b}$ (possibly $\vect{b}=\vect{c}$).
If $(\vect{c})_1=\ldots=(\vect{c})_k=0$, then $\vect{c} \sbelow \vect{a}$, contradicting the positivity of $f$. Therefore, 
$(\vect{c})_i=1$ for some $1\le i\le k$ and hence $\vect{c}$ is a minimal one corresponding to $x_i\in S$. Moreover, $\vect{c}$ is different from any minimal one $\vect{a}^j\in P$
because the $j$-th coordinate of $\vect{a}^j\in P$ is 1, while the $j$-th coordinate of $\vect{c}$ is 0.

A contradiction in both cases shows that there is no counterexamples to the statement of the lemma.
\end{proof}

\subsection{Canalyzing functions}
\label{sec:split}

\begin{lemma}\label{lem:split}
	Let $f = f(x_1, \ldots, x_n)$ be a positive canalyzing function 
	with $k \geq 0$ relevant variables. 
	Then the number of extremal points of $f$ is at least $k+1$. Moreover $f$
	has exactly $k+1$ extremal points if and only if $f$ is lro.
\end{lemma}
\begin{proof}
	The case $k=0$ is trivial, and therefore we assume that $k \geq 1$.

	Let $x_i$ be a variable of $f$ such that $f_{|x_i = 0} \equiv \vect{0}$
	(the case $f_{|x_i = 1} \equiv \vect{1}$ is similar). Let $f_0  =  f_{|x_i = 0}$ and $f_1  =  f_{|x_i = 1}$.
	Clearly, $x_i$ is a relevant variable of $f$, otherwise $f \equiv \vect{0}$, that is, $k=0$.
	Since every relevant variable of $f$ is relevant for at least one of the functions
	$f_0$ and $f_1$, we conclude that $f_1$ has $k-1$ relevant variables.
	
	The equivalence $f_0 \equiv \vect{0}$ implies that for every extremal point 
	$(\alpha_1, \ldots, \alpha_{i-1}, \alpha_{i+1}, \ldots, \alpha_n)$ of $f_1$, the corresponding
	point $(\alpha_1, \ldots, \alpha_{i-1}, 1, \alpha_{i+1}, \ldots, \alpha_n)$ is extremal for $f$.
	For the same reason, there is only one extremal point of $f$ with the $i$-th coordinate being
	equal to $0$, namely, the point with all coordinates equal to $1$, except for the $i$-th coordinate.
	Hence, $r(f) = r(f_1) + 1$. 
	\begin{enumerate}
		\item If $f_1$ is lro, then $f$ is also lro, since 
		$f$ can be expressed as $x_i \andl f_1$. 
		By the induction hypothesis $r(f_1) = k$ and therefore $r(f) = k+1$.
		
		\item If $f_1$ is not lro, then  
		 $f$ is also not lro, which is easy to see.
		By the induction hypothesis $r(f_1) > k$ and therefore $r(f) > k+1$.
	\end{enumerate}
\end{proof}

\subsection{Non-canalyzing functions with canalyzing restrictions}
\label{sec:ns1}

In this section, we study non-canalyzing positive functions such that for each variable $x_i$ both restrictions $f_{|x_i=0}$ and $f_{|x_i=1}$ are canalyzing.

First we remark that all variables of those functions are relevant.
Indeed, if such a function has an irrelevant variable then the function is canalyzing. 

\begin{claim}\label{cl:all_var_are_rel}
Let $f = f(x_1, \ldots, x_n)$ be a positive non-canalyzing function such that for each variable $x_i$ both restrictions $f_{|x_i=0}$ and $f_{|x_i=1}$ are canalyzing.
Then all variables of $f$ are relevant.
\end{claim}
\begin{proof}
Let $x_i$ be irrelevant, then $f_{|x_i=0} \equiv f_{|x_i=1}$.
But $f_{|x_i=0}$, $f_{|x_i=1}$ are canalyzing, hence there exists $p \in [n]$ such that 
$f_{|x_i=0,\,x_p=0} \equiv f_{|x_i=1,\,x_p=0} \equiv \vect{0}$ or 
$f_{|x_i=0,\,x_p=1} \equiv f_{|x_i=1,\,x_p=1} \equiv \vect{1}$.
In the former case $f_{|x_p=0} \equiv \vect{0}$,
in the latter case $f_{|x_p=1} \equiv \vect{1}$.
In any case $f$ is canalyzing. Contradiction.
\end{proof}

\begin{claim}\label{cl:com_var}
	Let $f = f(x_1, \ldots, x_n)$ be a positive non-canalyzing function such that for each variable $x_i$ both restrictions $f_{|x_i=0}$ and $f_{|x_i=1}$ are canalyzing.
Then for each $i$, 
\begin{itemize}
\item[(a)] there exists a maximal zero that contains $0$'s in exactly two coordinates one of which is $i$,  
\item[(b)] there exists a minimal one that contains $1$'s in exactly two coordinates one of which is $i$.
\end{itemize}
\end{claim}
\begin{proof}
  Fix an $i$ and denote $f_0 = f_{|x_i=0}$, $f_1 = f_{|x_i=1}$.
	Since $f_0$ is canalyzing, there exists $p \in [n]$ such that ${f_0}_{|x_p=0} \equiv \vect{0}$ or 
	${f_0}_{|x_p=1} \equiv \vect{1}$. We claim that the latter case is impossible. 
	Indeed, the positivity of $f$ and ${f_0}_{|x_p=1} \equiv \vect{1}$ imply ${f_1}_{|x_p=1} \equiv \vect{1}$, and therefore 
	$f_{|x_p=1} \equiv \vect{1}$.
	This contradicts the assumption that $f$ is non-canalyzing. 
	Thus, ${f_0}_{|x_p=0} \equiv \vect{0}$.
	Now we claim that the Boolean point $\vect{y}$ with exactly two $0$'s in coordinates $i$ and $p$ is a maximal zero.
	Indeed, if $f$ in at least one of three points above $\vect{y}$ is $0$, then, by positivity of $f$, 
	$f_{|x_i=0} = 0$ or $f_{|x_p=0}$, which contradicts the assumption that $f$ is non-canalyzing. 
	
	Similarly, one can show that ${f_1}_{|x_r=1} \equiv \vect{1}$ for some $r \in [n]$ implying that the Boolean point with exactly two $1$'s in coordinates $i$ and $r$ is a minimal one. 
\end{proof}

\begin{claim}\label{cl:VL}
	Let $f = f(x_1, \ldots, x_n)$ be a positive non-canalyzing function such that for each variable $x_i$ both restrictions $f_{|x_i=0}$ and $f_{|x_i=1}$ are canalyzing.
	Then there is a minimal one $\vect{\textup{y}}$ of Hamming weight $2$ such that $\overline{\vect{\textup{y}}}$ is a maximal zero, unless $n=4$ in which case $f$ has $6$ extremal points.
\end{claim}
\begin{proof}
Consider a graph $G_0$ (\textit{resp.} $G_1$) with vertex set $[n]$
every edge $ij$ of which represents a maximal zero (\textit{resp.} minimal one) 
that contains $0$'s (\textit{resp.} $1$'s) in exactly two coordinates $i$ and $j$. 
By Claim~\ref{cl:com_var}, 
every vertex in $G_0$ is covered by an edge and
every vertex in $G_1$ is covered by an edge.
From this it follows in particular that each graph $G_0$, $G_1$ has at least $\left\lceil n/2\right\rceil$ edges.


In terms of the graphs $G_0$ and $G_1$, the claim is equivalent to saying that $G_0$ and $G_1$ have a common edge.
It is not difficult to see that for $n\le 3$ the graphs $G_1$ and $G_0$ necessarily have a common edge. 
Let us show that this is also the case for $n\ge 5$. 
   
Assume that $G_0$ and $G_1$ have no common edges, i.e. every edge of $G_0$ is a non-edge (a pair of non-adjacent vertices) in $G_1$.
Let us prove that
\begin{itemize}
\item[(*)] every edge $ij$ of $G_0$ forms a vertex cover in $G_1$, i.e. every edge of $G_1$ shares a vertex with either $i$ or $j$ (and not with both according to our assumption). 
\end{itemize}
Indeed, let $ij$ be an edge of $G_0$ and assume that $G_1$ contains an edge $pq$ such that $p$ is different from $i,j$ and $q$ is different from $i,j$.
Then the minimal one corresponding to the edge $pq$ of $G_1$ is below the maximal zero corresponding to the  edge $ij$ of $G_0$. 
This contradicts the positivity of $f$ and proves (*).

Consider an edge $ij$ in $G_0$.
Since $n\ge 5$, then $G_0$ has at least $3$ edges, hence from (*) we get that at least one of $i,j$ 
covers at least two edges of $G_1$, say $i$ covers $ip$ and $iq$. 
Let $ps$ be an edge of $G_0$ covering $p$.
If $s\ne q$, then $ps$ does not cover the edge $iq$ of $G_1$ which contradicts to (*).
If $s=q$, let $t$ be any vertex different from $i,j,p,q$.
The vertex $t$ must be covered by some edge $tr$ in $G_1$.
If $r$ is different from $i,j$ then $tr$ does not cover $ij$ in $G_0$.
If $r$ is different from $p,q$ then $tr$ does not cover $pq$ in $G_0$.
In both cases we get a contradiction to (*), 
hence for $n\ge 5$ the graphs $G_0$ and $G_1$ necessarily have a common edge and hence the result follows in this case.

\medskip
It remains to analyze the case $n=4$. 
Up to renaming variables, 
the only possibility for $G_0$ and $G_1$ to avoid a common edge is for $G_0$ to have edges $12$ and $34$ and 
for $G_1$ to have edges $13$ and $24$.
In other words, $(0,0,1,1)$ and $(1,1,0,0)$ are maximal zeros and $(1,0,1,0)$ and $(0,1,0,1)$ are minimal ones. 
By positivity, this completely defines the function $f$,
except for two points $(0,1,1,0)$ and $(1,0,0,1)$. 
However, regardless of the value of $f$ in these points, both of them are extremal and hence $f$ has $6$ extremal points.
\end{proof}

\begin{claim}\label{cl:non-split} Let $f = f(x_1, \ldots, x_n)$ be a positive non-canalyzing function 
	such that for each variable $x_i$ both restrictions $f_{|x_i=0}$ and $f_{|x_i=1}$ are canalyzing.
Let $\vect{\textup{y}}$ be a minimal one of Hamming weight $2$ such that $\overline{\vect{\textup{y}}}$ is a maximal zero. 
Denote the two coordinates of $\vect{\textup{y}}$ containing $1$'s by $i$ and $s$, and 
let $f_0 = f_{|x_i=0}$ and $f_1 = f_{|x_i=1}$.
	\begin{enumerate}
		\item[(a)] Variable $x_s$ is relevant for
		both functions $f_0$ and $f_1$.
		
		\item[(b)] 
		If a point 
		$\vect{\textup{a}} = (\alpha_1, \ldots, \alpha_{i-1}, \alpha_{i+1}, \ldots, \alpha_{n}) \in B^{n-1}$ 
		is an extremal point of $f_{\alpha_i}$, $\alpha_i \in \{ 0, 1\}$,
		then 
		$\vect{\textup{a}}' = (\alpha_1, \ldots, \alpha_{i-1}, \alpha_i, \alpha_{i+1}, \ldots, \alpha_{n-1}) \in B^n$
		is an extremal point of $f$.
%
	\end{enumerate}
\end{claim}
\begin{proof}
First, we note that since $\vect{y}$ is a minimal one, ${f_1}_{|x_s=1} \equiv \vect{1}$.
Similarly, since $\overline{\vect{y}}$ is a maximal zero,  ${f_0}_{|x_s=0} \equiv \vect{0}$.

To prove (a), suppose to the contrary that $f_0$ does not depend on $x_s$. 
Then ${f_0}_{|x_s=1} \equiv {f_0}_{|x_s=0} \equiv \vect{0}$, and therefore
$f_0 \equiv \vect{0}$, which contradicts the assumption that $f$ is non-canalyzing.
Similarly, one can show that $x_s$ is relevant for $f_1$.
		
		
Now we turn to (b) and prove the statement for $\alpha_i = 1$. For $\alpha_i = 0$ the arguments are symmetric.

Assume first that $\alpha_s=1$. 
Since $\vect{y}$ is a minimal one,  we have $f_1(\vect{b})=1$ for all $\vect{b}=(\beta_1,\dots,\beta_{i-1},\beta_{i+1},\dots,\beta_n)$
with $\beta_s = 1$. Due to the extremality of ${\bf a}$, all its components besides $\alpha_s$
are zeros. It follows that $\vect{a}'=\vect{y}$, which is a minimal one by assumption.

It remains to assume that $\alpha_s=0$.
Let $\vect{a}$ be a maximal zero for the function $f_1$. 
If $\vect{a}'$ is not a maximal zero for $f$, then there is $\vect{a}''\succ\vect{a}'$ with $f(\vect{a}'')=0$. 
Since $\vect{a}''\succ\vect{a}'$ and $\alpha_i = 1$, the $i$-th component of $\vect{a}''$ is $1$. 
By its removal, we obtain a zero of $f_1$ that is strictly above $\vect{a}$ in contradiction to the minimality of the latter.

Let $\vect{a}$ be a minimal one for the function $f_1$. 
If $\vect{a}'$ is not a minimal one for $f$, then there is $\vect{a}''\prec\vect{a}'$ with $f(\vect{a}'')=1$.
The $i$-th component of $\vect{a}''$ must be $0$, since otherwise by its removal we obtain a one for $f_1$ strictly below $\vect{a}$.
Also, the $s$-th component of $\vect{a}''$ must be $0$, since this component equals $0$ in $\vect{a}$. 
But then $\vect{a}'' \preceq \overline{\vect{y}}$ with $f(\vect{a}'')=1$ and $f(\overline{\vect{y}})=0$, a contradiction.
\end{proof}

\begin{lemma}\label{lem:non-split-split-projection}
	Let $f = f(x_1, \ldots, x_n)$ be a positive non-canalyzing function 
	such that for each variable $x_i$ both restrictions $f_{|x_i=0}$ and $f_{|x_i=1}$ are canalyzing.
	Then the number of extremal points of $f$ is at least $n+2$.
\end{lemma}
\begin{proof}
By Claim~\ref{cl:VL} we may assume that there is a minimal one $\vect{y}$ 
that contains $1$'s in exactly two coordinates, say $i$ and $s$, such that $\overline{\vect{y}}$ is a maximal zero.
Denote $f_0 = f_{|x_i=0}$ and $f_1 = f_{|x_i=1}$.

	Let $P$, $P_0$, and $P_1$ be the sets of relevant variables of $f, f_0$, and $f_1$,
	respectively. By Claim~\ref{cl:all_var_are_rel}, $P$ is the set of all variables. 
	Since any relevant variable of $f$ is relevant for at least one 
	of the functions $f_0$, $f_1$ and, by Claim \ref{cl:non-split} (a), $x_s$ is a relevant 
	variable of both of them, we have
	$$
		n = |P| \leq |P_0 \cup P_1| + 1 = |P_0| + |P_1| - |P_0 \cap P_1| + 1 \leq |P_0| + |P_1|.
	$$
	By Lemma~\ref{lem:split}, $r(f_0) \geq |P_0|+1$, $r(f_1) \geq |P_1|+1$. 
	Finally, by Claim \ref{cl:non-split} (b) the number $r(f)$ of 
	extremal points of $f$ is at least $r(f_0) + r(f_1) \geq |P_0| + |P_1| + 2 \geq n + 2$.
\end{proof}

\subsection{Non-canalyzing functions containing a non-canalyzing restriction}
\label{sec:ns2}

Due to Lemmas \ref{lem:split} and \ref{lem:non-split-split-projection}  
it remains to show the bound for a positive non-canalyzing function $f = f(x_1, \ldots, x_n)$ 
such that for some $i \in [n]$ at least one of $f_0 = f_{|x_i=0}$ and $f_1 = f_{|x_i=1}$ is non-canalyzing.
Let $k$ be the number of relevant variables of $f$ and let us prove that the number of extremal points of $f$ is at least $k+2$.


Consider two possible cases:
\begin{enumerate}
	\item[(a)] $x_i$ is a irrelevant variable of $f$;
	\item[(b)] $x_i$ is a relevant variable of $f$.
\end{enumerate}

In case (a) the function $f_{|x_i=0} \equiv f_{|x_i=1}$ is non-canalyzing and has the same number of extremal points and 
the same number of relevant variables as $f$. By induction, the number of extremal points of $f$ is at least $k+2$.

Now let us consider case (b).
Assume without loss of generality that $i=n$, and
let $f_0 = f_{|x_n=0}$ and $f_1 = f_{|x_n=1}$. 
We assume that $f_0$ is non-canalyzing and prove that $f$ has at least $k+2$ extremal points, where
$k$ is the number of relevant variables of $f$.
The case when $f_0$ is canalyzing, but $f_1$ is non-canalyzing is proved similarly.

Let us denote the number of relevant variables of $f_0$ by $m$. Clearly, $1 \leq m \leq k-1$.
Exactly $k-1-m$ of $k$ relevant variables of $f$ are irrelevant for the function $f_0$.
Note that these $k-1-m$ variables are necessarily relevant for the function $f_1$.
By the induction hypothesis, the number $r(f_0)$ of extremal points of $f_0$ is at least $m+2$.

We introduce the following notation:
\begin{enumerate}
	\item[$C_0$] -- the set of maximal zeros of $f$ corresponding to $x_n$;
	\item[$P_0$] -- the set of all other maximal zeros of $f$, i.e., $P_0 = Z^f \setminus C_0$;
	\item[$C_1$] -- the set of minimal ones of $f$ corresponding to $x_n$;
	\item[$P_1$] -- the set of all other minimal ones of $f$, i.e., $P_1 = U^f \setminus C_1$.
\end{enumerate}

For a set $A \subseteq B^n$ we will denote by $A^*$ the restriction of $A$ to the first $n-1$
coordinates, i.e., 
$A^* = \{ (\alpha_1, \ldots, \alpha_{n-1})~|~(\alpha_1, \ldots, \alpha_{n-1},\alpha_{n}) \in A \text{ for some }
\alpha_n \in \{0,1\} \}$.

By definition, the number of extremal points of $f$ is 
\begin{equation}\label{eq:r_f}
	r(f) = |C_0| + |P_1| + |C_1| + |P_0| = |C_0^*| + |P_1^*| + |C_1^*| + |P_0^*|.
\end{equation}

We want to express $r(f)$ in terms of the number of extremal points of $f_0$ and $f_1$. For this
we need several observations. First, we observe that if $(\alpha_1, \ldots, \alpha_{n-1}, \alpha_n)$ is an extremal point 
 for $f$, the point $(\alpha_1, \ldots, \alpha_{n-1})$
is extremal for $f_{\alpha_n}$. Furthermore, we have the following straightforward claim.

\begin{claim}\label{cl:p1_p0}
	$P_1^*$ is the set of minimal ones of $f_0$ and $P_0^*$ is the set of maximal zeros of $f_1$.
\end{claim}

In contrast to the minimal ones of $f_0$, the set of maximal zeros of $f_0$ in addition to the points in $C_0^*$
may contain extra points, which we denote by $N_0^*$. In other words, $Z^{f_0} = C_0^* \cup N_0^*$.
Similarly, besides $C_1^*$, the set of minimal ones of $f_1$ may contain additional points, which we denote
by $N_1^*$. That is, $U^{f_1} = C_1^* \cup N_1^*$.

\begin{claim}\label{cl:n0_p0}
	The set $N_0^*$ is a subset of the set $P_0^*$ of maximal zeros of $f_1$.
	The set $N_1^*$ is a subset of the set $P_1^*$ of minimal ones of $f_0$.
\end{claim}
\begin{proof}
We will prove the first part of the statement, the second one is proved similarly.
Suppose to the contrary that there exists a point 
$\vect{a} = (\alpha_1, \ldots, \alpha_{n-1}) \in N_0^* \setminus P_0^*$, which is a maximal zero for
$f_0$, but is not a maximal zero for $f_1$.
Notice that $f_1(\vect{a}) = 0$, as otherwise $(\alpha_1, \ldots, \alpha_{n-1}, 0)$ would be a maximal zero for
$f$, which is not the case, since $\vect{a} \notin C_0^*$.
Since $\vect{a}$ is not a maximal zero for $f_1$, there exists a maximal zero $\vect{b} \in B^{n-1}$
for $f_1$ such that $\vect{a} \sbelow \vect{b}$.
But then we have $f_0(\vect{b}) = 1$ and $f_1(\vect{b}) = 0$, which contradicts the positivity
of function $f$.
\end{proof}

From Claim \ref{cl:p1_p0} we have 
$r(f_0) = |Z^{f_0} \cup U^{f_0}| = |C_0^*| + |N_0^*| + |P_1^*|$, which together with (\ref{eq:r_f}) 
and Claim \ref{cl:n0_p0} imply

\begin{equation}\label{eq:rf}
\begin{split}
	r(f) & = |C_0^*| + |P_1^*| + |C_1^*| + |P_0^*| = |C_0^*| + |P_1^*| + |C_1^*|  + |N_0^*| + 
	|P_0^* \setminus N_0^*| \\ 
	& = r(f_0) + |C_1^*| + |P_0^* \setminus N_0^*|.
\end{split}
\end{equation}

Using the induction hypothesis we conclude that $r(f) \geq m+2 + |C_1^*| + |P_0^* \setminus N_0^*|$.
To derive the desired bound $r(f) \geq k+2$, in the rest of this section we show that 
$C_1^* \cup P_0^* \setminus N_0^*$ contains at least $k-m$ points.

\begin{claim}\label{cl:proper_pairs}
	Let $x_i$, $i \in [n-1]$, be a relevant variable for $f_1$, which is irrelevant for $f_0$.
	Then every maximal zero for $f_1$ corresponding to $x_i$ belongs to $P_0^* \setminus N_0^*$ 
and every minimal one for $f_1$ corresponding to $x_i$ belongs to $C_1^*$.
\end{claim}
\begin{proof}
Let $\vect{x} \in N_0^*$ and assume $(\vect{x})_i = 0$. 
Then by changing in $\vect{x}$ the $i$-th coordinate from $0$
to $1$ we obtain a point $\vect{x}'$ with $f_0(\vect{x}') = 1 \neq f_0(\vect{x})$, 
since $\vect{x}$ is a maximal zero for $f_0$.
This contradicts the assumption that $x_i$ is irrelevant for $f_0$.
Therefore, $(\vect{x})_i = 1$ and hence no maximal zero for $f_1$ corresponding to $x_i$ belongs to $N_0^*$, i.e.
every maximal zero for $f_1$ corresponding to $x_i$ belongs to $P_0^* \setminus N_0^*$ 

Similarly, one can show that no minimal one for $f_1$ corresponding to $x_i$ belongs to $N_1^*$, i.e.
every minimal one for $f_1$ corresponding to $x_i$ belongs to $C_1^*$.
\end{proof}

Recall that there are exactly $k - 1 - m$ variables that are relevant for $f_1$ and irrelevant for $f_0$.
Lemma~\ref{lem:acyclic} implies that there are at least $k-m$ extremal points for $f_1$ corresponding to these variables. 
By Claim~\ref{cl:proper_pairs}, all these points belong to the set $C_1^* \cup P_0^* \setminus N_0^*$.
This conclusion establishes the main result of this section.

\begin{lemma}\label{lem:non-split-without-split-projection}
	Let $f = f(x_1, \ldots, x_n)$ be a positive non-canalyzing function 
	with $k$ relevant variables such that for some $i \in [n]$ at least one of the restrictions $f_0 = f_{|x_i=0}$ and $f_1 = f_{|x_i=1}$ is non-canalyzing.
	Then the number of extremal points of $f$ is at least $k+2$.
\end{lemma}


\section{Chow and read-once functions}
\label{sec:Chow}

In this section, we look at the intersection of the classes of Chow and read-once functions
and show that this is precisely the class of lro functions. Thus, our result generalizes a result from  \cite{lists} showing that the class of lro functions 
is the intersection of the classes of read-once and threshold functions.  


There are two read-once functions that play a crucial role in our characterization of read-once Chow functions:
$$
	g_1 = g_1(x,y,z,u) = (x \orl y) \andl (z \orl u);
$$
$$
	g_2 = g_2(x,y,z,u) = (x \andl y) \orl (z \andl u).
$$


\begin{lemma}\label{cl:min-ro-non-chow}
	Functions $g_1, g_2$ and all the functions obtained from them by negation of variables are not Chow. 
\end{lemma}
\begin{proof}
	Function $g_1$ is not Chow, because $g_1$ is different from  $(x \orl z) \andl (y \orl u)$ (e.g. they have different values at the point $x=1$, $y=0$, $z=1$, $u=0$),
but both functions  have the same  Chow parameters $(6,6,6,6,9)$.
In a similar way, one can show that neither $g_2$ nor any function obtained from $g_1$ or $g_2$ by negation of variables is Chow. 
\end{proof}


The following lemma shows that the class of Chow functions is closed under taking restrictions.
\begin{lemma}\label{lem:chow-closed}
	If $f(x_1, \dots, x_n)$ is a Chow function, then any restriction of $f$ is also Chow.
\end{lemma}
\begin{proof}
Suppose to the contrary that $f$ has a restriction which is not a Chow function, namely,
$$
	g = g(x_{i_{k+1}}, \dots, x_{i_n}) := f_{x_{i_{1}}= \alpha_{1}, \dots, x_{i_k} = \alpha_{k}},
$$
for some $i_1, \dots, i_n \in [n]$, $\alpha_1, \dots, \alpha_k \in \{0,1\}$ and $g$ is not a Chow function.
Then there exists a function $g' = g'(x_{i_{k+1}}, \dots, x_{i_n})$ with the same Chow parameters as $g$. 
We define function $f'(x_1, \dots, x_n)$ as follows:

$$
f'(x_1, \dots, x_n) = 
\begin{cases} 
	f(x_1, \dots, x_n) 			& \mbox{if } (x_{i_1}, \dots, x_{i_k}) \neq (\alpha_1, \dots, \alpha_k), \\ 
	g'(x_{k+1}, \dots, x_n) 	& \mbox{if } (x_{i_1}, \dots, x_{i_k}) = (\alpha_1, \dots, \alpha_k).
\end{cases}
$$

Since $w(g) = w(g')$, we conclude that $w(f) = w(f')$.  
Similarly, for every $i \in \{i_{k+1}, \dots, i_n\}$, the equality $w_i(g)=w_i(g')$ implies
that $w_i(f)=w_i(f')$.
Consequently, $f$ and $f'$ have the same Chow parameters, which contradicts the fact that $f$ is Chow.
\end{proof}


\begin{lemma}\label{cl:non-lro-non-split}
Any canalyzing read-once function $f$, which is not lro,  has a non-constant non-canalyzing read-once function as a restriction.
\end{lemma}
\begin{proof}
Let $f$ be a minimum counterexample to the claim. 
Since $f$ is canalyzing, there exists $\alpha, \beta \in \{ 0, 1 \}$ such that $f_{|x_i=\alpha} \equiv \beta$. 
We assume that $\alpha = \beta = 1$, i.e. $f_{|x_i=1} \equiv \vect{1}$, in which case $f = x_i \orl f_{|x_i=0}$ (the other cases are similar).

Clearly, $f_{|x_i=0}$ is read-once, since any restriction of a read-once function is read-once. 
Also,  $f_{|x_i=0}$ is not lro, since otherwise  $f$ is lro, and hence  $f_{|x_i=0}$ is not a constant function. 
Since $f$ is a counterexample,  $f_{|x_i=0}$ is canalyzing and has no non-constant non-canalyzing read-once restrictions.
But then we have a contradiction to the minimality of $f$.
\end{proof} 

\begin{theorem}\label{th:lro-chow}
	For a read-once function $f$ the following statements are equivalent:
	\begin{enumerate}
		\item[$(1)$] $f$ is an lro function;
		\item[$(2)$] $f$ is a Chow function;
		\item[$(3)$] $f$ does not have $g_1$ or $g_2$ or any function obtained from $g_1$ or $g_2$ by negation of variables as a restriction.
	\end{enumerate}
\end{theorem}
\begin{proof}
It is known that all lro functions are threshold \cite{lists} and all threshold functions are Chow \cite{Chow1961}.
Therefore, $(1)$ implies $(2)$.

To prove that $(2)$ implies $(3)$, we observe that by Lemma~\ref{lem:chow-closed} any restriction of $f$ is Chow. 
This together with Lemma~\ref{cl:min-ro-non-chow} imply the conclusion.
	

	Finally, to prove that $(3)$ implies $(1)$, we assume that $f$ is positive 
and show that if $f$ is non-lro, then it has as a restriction at least one of the functions $g_1$ and $g_2$ (in the case
of a non-positive function, similar arguments show that $f$ contains as a restriction a function obtained from $g_1$ or $g_2$ by possibly negating some variables).
	Also, without loss of generality we assume that $f$ is non-canalyzing, otherwise we would consider a non-constant
	non-canalyzing restriction of $f$ which is guaranteed by Lemma~\ref{cl:non-lro-non-split}.
	
Since $f$ is a read-once function, there exist read-once functions $f_1$ and $f_2$ such that 
	either $f = f_1 \andl f_2$ or $f = f_1 \orl f_2$ and the sets of relevant variables of $f_1$ and $f_2$ 
	are disjoint.
	We let $f = f_1 \orl f_2$, since the other case can be proved similarly.
	Suppose, one of the functions $f_1$ and $f_2$, say $f_1$, does not contain a conjunction in its 
	read-once formula. Then for any relevant variable $x_i$ of $f_1$ we have $f_{|x_i=1} \equiv \vect{1}$,
	which contradicts the assumption that $f$ is non-canalyzing.
	Hence, both $f_1$ and $f_2$ necessarily contain conjunctions in their read-once formulas.
	This means that there exist 
	$i_1, \dots, i_{n} \in [n]$, $\alpha_5,\dots,\alpha_{n} \in \{0,1\}$ such that 
	$$
		f_{1|x_{i_5}=\alpha_5,\dots,x_{i_k}=\alpha_k}= x_{i_1} \andl x_{i_2}, \text{ and }
	$$
	$$
		f_{2|x_{i_{k+1}}=\alpha_{k+1},\dots,x_{i_{n}}=\alpha_{n}}= x_{i_3} \andl x_{i_4},
	$$
	where $\{x_{i_5}, \dots, x_{i_k}\}$ and $\{x_{i_{k+1}}, \dots, x_{i_{n}}\}$ are the sets of relevant variables 
	of the functions $f_1$ and $f_2$, respectively.
	Consequently
	\begin{equation}
		\begin{split}
			& f_{|x_{i_5}=\alpha_5, \dots, x_{i_{n}}=\alpha_{n}} = \\
			& f_{1|x_{i_5}=\alpha_5,\dots,x_{i_k}=\alpha_k} \orl 
			f_{2|x_{i_{k+1}}=\alpha_{k+1}, \dots, x_{i_{n}}=\alpha_{n}} = \\
			& (x_{i_1} \andl x_{i_2}) \orl (x_{i_3} \andl x_{i_4}).
		\end{split}
	\end{equation}
\end{proof}

\section{Threshold functions and specification number}
\label{sec:threshold}

In this section, we turn to threshold functions and characterize the class of lro functions within this universe by a set $\cal G$ of minimal functions which are not linear read-once
(Section~\ref{subsec-counerexample2}). All functions in $\cal G$ depending on $n$ variables  have specification number $2n$, which can be viewed as an argument supporting Conjecture~\ref{con:conjecture}.
Nevertheless, in Section~\ref{subsec-counerexample} we disprove the conjecture.

\subsection{Minimal non-lro threshold functions}
\label{subsec-counerexample2}

For $n\ge 3$, denote by $g_n$ the function defined by its DNF
$$
g_n = x_1x_2 \orl x_1x_3 \orl \dots \orl x_1x_n \orl x_2\dots x_n.
$$

\begin{lemma}\label{lem:g}
	For any $n \ge 3$,
	the function $g_n$ is positive, non-lro, threshold function, depending on all its variables,
	 and the specification number of $g_n$ is $2n$.
\end{lemma}
\begin{proof}
	Clearly, $g_n$ depends on all its variables.
	Furthermore, $g_n$ is positive, since its DNF contains no negation of a variable.
	Also, it is easy to verify that $g_n$ is not canalyzing, and therefore 
	$g$ is non-lro.
	
	Now, we claim that the CNF of $g_n$ is
	$$
		(x_1 \vee x_2)(x_1 \vee x_3)\dots (x_1 \vee x_n)(x_2 \vee x_3 \vee \dots \vee x_n).
	$$
	Indeed, the equivalence of the DNF and CNF can be directly checked by expanding the latter
	and applying the absorption law:
	\begin{equation*}
		\begin{split}
		       & (x_1 \vee x_2)(x_1 \vee x_3)\dots (x_1 \vee x_n)(x_2 \vee x_3 \vee \dots \vee x_n) \\
			& = (x_1 \vee x_2x_3\dots x_n)(x_2 \vee x_3 \vee \dots \vee x_n) \\
			& = x_1 x_2 \vee x_1 x_3 \vee \dots \vee x_1x_n \vee x_2x_3\dots x_n.
		\end{split}
	\end{equation*}
	From the DNF and the CNF of $g_n$ we retrieve the minimal ones
	$$
	\begin{array}{r}
		\vect{x}_1 = (1,1,0,\dots,0),\\
		\vect{x}_2 = (1,0,1,\dots,0),\\
		\dotfill                         \\
		\vect{x}_{n-1} = (1,0,0,\dots,1),\\
		\vect{x}_n = (0,1,1,\dots,1),\\
	\end{array}	
	$$
	and maximal zeros of $g_n$
	$$
	\begin{array}{r}
    \vect{y}_1 = (0,0,1,\dots,1),\\
    \vect{y}_2 = (0,1,0,\dots,1),\\
		\dotfill                             \\
		\vect{y}_{n-1} = (0,1,1,\dots,0),\\
		\vect{y}_n = (1,0,0,\dots,0),\\
	\end{array}
	$$
	respectively (see Theorems 1.26, 1.27 in \cite{CH11}).
	
	It is easy to check that all minimal ones $\vect{x}_1,\vect{x}_2,\dots,\vect{x}_n$ 
	satisfy the equation 
	$$
		(n-2)x_1 + x_2+x_3+\dots+x_n = n-1,
	$$
	and all maximal zeros $\vect{y}_1, \vect{y}_2, \dots, \vect{y}_n$
	satisfy the equation 
	$$
		(n-2)x_1 + x_2+x_3+\dots+x_n = n-2.
	$$
	Hence $(n-2)x_1 + x_2+x_3+\dots+x_n \le n-1$ is a threshold inequality representing the function $g_n$.

	For any positive threshold function $f$, which depends on all its variables, the set of
 	its extremal points specifies $f$. Furthermore, every essential point of $f$ must belong to each specifying set.
 	Therefore, all essential points of $g_n$ are extremal.
	On the other hand, by Theorem~\ref{th:ess-all}, 
	all extremal points of $g_n$ are essential and therefore, by Theorem~\ref{cl:ess-sigma}, $\sigma_{\mathcal{H}_n}(g_n) = 2n$.
\end{proof}

It is not difficult to see that $g_n$ ($n\ge 3$) is a {\it minimal} threshold function which is not lro,
i.e. any restriction of $g_n$ ($n\ge 3$) is an lro function. Moreover, the same is true for any function obtained from $g_n$ ($n\ge 3$) by negation of variables, 
since the negation of a variable of a threshold function results in a threshold function. We denote the set of all these minimal functions by 
$\cal G$ and show in what follows that there are no other minimal threshold functions which are not lro.

\begin{theorem} \label{prop:g}
A threshold function $f$ is lro if and only if it does not contain any function from $\cal G$ as a restriction.
\end{theorem}
\begin{proof}
Stetsenko  proved in \cite{Stetsenko} that the set of all, up to renaming and negation of variables, minimal not read-once functions consists of 
the following functions:
$$
\begin{array}{ll}
g_n(x_1,\dots,x_n) = x_1(x_2 \orl \dots \orl x_n) \orl x_2 \dots x_n                            & (n\ge 3);\\
h^1_n(x_1,\dots,x_n) = x_1\dots x_n \orl \overline{x_1} \dots \overline{x_n}                    & (n\ge 2);\\
h^2_n(x_1,\dots,x_n) = x_1(x_2 \orl x_3 \dots x_n) \orl x_2 \overline{x_3} \dots \overline{x_n} & (n\ge 3);\\
h^3(x_1, \dots, x_5) = x_1(x_3x_4 \orl x_5) \orl x_2(x_3 \orl x_4x_5);\\
h^4(x_1, \dots, x_4) = x_1(x_2 \orl x_3) \orl x_3x_4.
\end{array}
$$
Let us show that all functions in this list, except for $g_n$, are $2$-summable, hence are not threshold.
\begin{itemize}
\item For the function $h^1_n$ we have:
$$
\begin{array}{l}
h^1_n(1,0,\ldots,0)=h^1_n(0,1,\ldots,1)=0,\\
h^1_n(0,0,\ldots,0)=h^1_n(1,1,\ldots,1)=1
\end{array}
$$
and
$$
(1,0,\ldots,0)+(0,1,\ldots,1)=(0,0,\ldots,0)+(1,1,\ldots,1).
$$

\item For the function $h^2_n$ we have:
$$
\begin{array}{l}
h^2_n(1,0,0,\ldots,0)=h^2_n(0,1,1,\ldots,1)=0,\\
h^2_n(0,1,0,\ldots,0)=h^2_n(1,0,1,\ldots,1)=1
\end{array}
$$
and
$$
(1,0,\ldots,0)+(0,1,\ldots,1)=(0,1,0,\ldots,0)+(1,0,1,\ldots,1).
$$

\item For the function $h^3$ we have:
$$
\begin{array}{l}
h^3(0,0,1,1,1)=h^3(1,1,0,0,0)=0,\\
h^3(0,1,1,0,0)=h^3(1,0,0,1,1)=1
\end{array}
$$
and
$$
(0,0,1,1,1)+(1,1,0,0,0)=(0,1,1,0,0)+(1,0,0,1,1).
$$

\item For $h^4$ we have:
$$
\begin{array}{l}
h^4(1,0,0,1)=h^4(0,1,1,0)=0,\\
h^4(1,1,0,0)=h^4(0,0,1,1)=1
\end{array}
$$
and 
$$
(1,0,0,1)+(0,1,1,0)=(1,1,0,0)+(0,0,1,1).
$$
\end{itemize}
Since the functions $h^1_n,h^2_n,h^3,h^4$ are not threshold, $f$ does not contain any of them or any function obtained from them by negation of variables as a restriction.
If, additionally, $f$ does not contain any function from $\cal G$, then  $f$ is read-once and hence is lro. If $f$ contains a function from $\cal G$ as a restriction,
then $f$ is not read-once and hence is not lro. 
\end{proof}

\subsection{Non-lro threshold functions with minimum specification number}
\label{subsec-counerexample}

%

\begin{theorem}\label{prop:1}
	For a natural number $n \geq 4$, let $f_n = f(x_1, \ldots, x_n)$ be a function defined by its DNF
	$$
		x_1 x_2 \vee x_1 x_3 \vee \dots \vee x_1x_{n-1} \vee x_2x_3\dots x_n.
	$$
	Then $f_n$ is positive, non-lro, threshold function, depending on all its variables,
	 and the specification number of $f_n$ is $n+1$.
\end{theorem}
\begin{proof}
	Clearly, $f_n$ depends on all its variables,
	it is positive, not canalyzing, and therefore $f$ is non-lro.
	
	It is easy to verify that CNF of $f_n$ is
	$$
		(x_1 \vee x_2)(x_1 \vee x_3)\dots (x_1 \vee x_n)(x_2 \vee x_3 \vee \dots \vee x_{n-1}).
	$$
	From the DNF and the CNF of $f_n$ we retrieve the minimal ones
	$$
	\begin{array}{r}
		\vect{x}_1 = (1,1,0,\dots,0,0),\\
		\vect{x}_2 = (1,0,1,\dots,0,0),\\
		\dotfill                         \\
		\vect{x}_{n-2} = (1,0,0,\dots,1,0),\\
		\vect{x}_{n-1} = (0,1,1,\dots,1,1)\phantom{,}\\
	\end{array}	
	$$
	and maximal zeros of $f_n$
	$$
	\begin{array}{r}
    		\vect{y}_1 = (0,0,1,\dots,1,1),\\
    		\vect{y}_2 = (0,1,0,\dots,1,1),\\
		\dotfill                             \\
		\vect{y}_{n-2} = (0,1,1,\dots,0,1),\\
		\vect{z}_1 = (0,1,1,\dots,1,0),\\
		\vect{z}_2 = (1,0,0,\dots,0,1).\\
	\end{array}
	$$
	It is easy to check that all minimal ones $\vect{x}_1,\vect{x}_2,\dots,\vect{x}_{n-1}$ 
	belong to the hyperplane 
	$$
		(2n-5)x_1 + 2(x_2+x_3+\dots+x_{n-1})+x_n = 2n-3,
	$$
	the maximal zeros $\vect{y}_1, \vect{y}_2, \dots, \vect{y}_{n-2}$
	belong to the hyperplane 
	$$
		(2n-5)x_1 + 2(x_2+x_3+\dots+x_{n-1})+x_n = 2n-5,
	$$
	and the maximal zeros $\vect{z}_1, \vect{z}_2$ belong to the hyperplane 
  $$
		(2n-5)x_1 + 2(x_2+x_3+\dots+x_{n-1})+x_n = 2n-4.
	$$
  Hence, $(2n-5)x_1 + 2(x_2+x_3+\dots+x_{n-1})+x_n \leq 2n-4$ is a threshold inequality representing $f_n$.
	
	
	As in the proof of Lemma~\ref{lem:g} we conclude that every essential point of $f_n$ is extremal.
	However, the extremal points $\vect{y}_1, \vect{y}_2, \dots, \vect{y}_{n-2}$ are not essential. 
	To show this, suppose to the contrary that there exists a threshold function $d_i$ that differs from $f_n$ only 
	in the point $\vect{y}_i$, $i \in [n-2]$, i.e., $d_i(\vect{y}_i)=1$ and $d_i(\vect{x})=f_n(\vect{x})$ 
	for every $\vect{x}\ne \vect{y}_i$.
	Then $\vect{x}_i + \vect{y}_i = \vect{z}_1 + \vect{z}_2$, and hence $d_i$ is 2-summable. 
	Therefore, by Theorem~\ref{th:asum}, $d_i$ is not a threshold function, a contradiction.

	Now Theorems \ref{cl:ess-sigma} and \ref{th:ess-all} imply that all
	the remaining $n+1$ extremal points $\vect{x}_1,\dots,\vect{x}_{n-1}$, $\vect{z}_1,\vect{z}_2$
	are essential, and therefore $\sigma_{\mathcal{H}_n}(f_n) = n+1$.
\end{proof}

\section{Conclusion}
In this paper we proved a number of results related to the class of linear read-once functions. 
We also  showed the existence of positive threshold Boolean functions of $n$ variables,
which are not linear read-once and for which the specification number is at its lowest
bound, $n + 1$. This leaves open the problem of characterizing the set of all such functions.
We observe that this set is not closed under taking restrictions. In particular, the functions 
described in Theorem~\ref{prop:1} contain, as restrictions, the functions from the set $\cal G$.
This example also shows that specification number is not monotone with respect to restrictions,
i.e. by restricting a function we can increase the specification number. All these remarks 
suggest that the problem of characterizing the set of all threshold functions with minimum 
value of specification number is highly non-trivial.

\paragraph*{Acknowledgment} 
The authors are grateful to Dmitry Chistikov for the useful discussions on the topic and for drawing their attention to the work 
of V.~Stetsenko \cite{Stetsenko}.

The work of Vadim Lozin and Nikolai Yu. Zolotykh was supported by the Russian Science Foundation Grant No. 17-11-01336.

\end{document}